\documentclass[12pt]{amsart}
\usepackage[margin=1.1in]{geometry}
\usepackage{amssymb,amsthm,hyperref}
\usepackage{epsfig}
\usepackage[arrow]{xy}
\hypersetup{hidelinks}

\newtheorem{lem}{Lemma}
\newtheorem{lemma}[lem]{Lemma}

\newtheorem{theorem}[lem]{Theorem}
\newtheorem{definition}[lem]{Definition}

\newtheorem{conjecture}[lem]{Conjecture}

\theoremstyle{remark}
\newtheorem{remark}{Remark}

\numberwithin{lem}{section}

\numberwithin{equation}{section}

\def\A{{\mathcal A}}
\def\Z{{\mathbb Z}}
\def\L{{\mathcal L}}
\def\x{{\bf x}}
\def\N{{\mathcal N}}

\newcommand{\LP}{\L^{+}_{\N}}

\hypersetup{
  pdftitle={Positivity for graph Laurent phenomenon algebras},
  pdfauthor={Ross Glew},
  pdfsubject={Positivity for graph Laurent phenomenon algebras}
}

\title{Positivity for graph Laurent phenomenon algebras}

\author{Ross Glew}\address
 {School of Natural Sciences\\ Institute for Advanced Study\\ Princeton\\ NJ 08540 USA.}

\begin{document}

\begin{abstract}
Laurent phenomenon algebras were introduced by Lam and Pylyavskyy as a generalisation of cluster algebras. For a family associated with directed graphs, they showed that every cluster variable is a Laurent polynomial in the variables of any cluster and conjectured that these expansions have nonnegative coefficients. We prove this conjecture for all directed graphs without loops or multiple edges with respect to an arbitrary seed cluster. 
\end{abstract}

\maketitle

\section{Introduction}
The cluster algebras of Fomin and Zelevinsky~\cite{CA1}
are generated by \emph{cluster variables}, which are grouped into sets called \emph{clusters}. An operation known as \emph{mutation} replaces one variable in a cluster with another, allowing one to generate new clusters from old.

Laurent phenomenon (LP) algebras were introduced in~\cite{LP} as a generalisation of cluster algebras, with a more general notion of mutation. A particularly well-behaved family associated with directed graphs was subsequently introduced in~\cite{LP2}.

These graph LP algebras share several important features of cluster algebras. Perhaps the most fundamental is the Laurent phenomenon: every cluster variable can be expressed as a Laurent polynomial in the variables of any given cluster~\cite{LP2}. These expansions were further conjectured to have the following important property~\cite[Conjecture~7.1(2)]{LP2}. 
\begin{conjecture}
For every directed graph without loops or multiple edges, each cluster variable of the associated graph LP algebra is a Laurent polynomial with nonnegative coefficients in the variables of any cluster.
\end{conjecture}
Partial results for tree graphs were previously obtained in~\cite{BCKZ,BCKZ2}. Related aspects of graph LP algebras have also been studied in~\cite{GM,DKW}. 

The aim of this paper is to prove the Laurent positivity conjecture in full generality for graph LP algebras. The proof relies on two
manifestly positive identities involving products of cluster variables:
an $XY$-identity already established in~\cite{LP2} and a new
$YY$-identity\footnote{The $YY$-identity was suggested by ChatGPT following discussions on related topics.} proved here. Taken together, these two identities lead to a straightforward inductive proof of positivity.

The paper is organised as follows. In Section~\ref{sec:nested}, we introduce strongly connected subsets and nested collections, which characterise the clusters of graph LP algebras. In Section~\ref{sec:graph-lp}, we define the cluster variables explicitly in terms of acyclic functions on the graph. In Section~\ref{sec:id}, we present two subtraction-free identities that form the basis of the positivity proof. In Section~\ref{sec:positivity}, we use these identities to prove the Laurent positivity.

\section{Nested sets}\label{sec:nested}

Let $\Gamma$ be a directed graph on $[n]=\{1,\ldots,n\}$
without loops or multiple edges. A nonempty subset
$I\subseteq[n]$ is \emph{strongly connected} if its induced
subgraph $\Gamma[I]$ is strongly connected. For
$I\subseteq[n]$, let $\operatorname{SCC}(I)$ denote the strongly connected
components of $\Gamma[I]$, with $\operatorname{SCC}(\varnothing)=\varnothing$. We will make use of the notation $Si=S\cup\{i\}$ and $Sij=S\cup\{i,j\}$.

\begin{definition}
A collection $\N$ of strongly connected subsets is nested if
\begin{enumerate}
\item for any $I,J\in\N$, either $I\subseteq J$,
      $J\subseteq I$, or $I\cap J=\varnothing$;
\item any pairwise disjoint members of $\N$ are the strongly
      connected components of their union.
\end{enumerate}
\end{definition}
The \emph{support} of $\N$ is defined as $N=\bigcup_{I\in\N}I$.
We say that $\N$ is \emph{maximal on its support}, or simply
\emph{maximal}, if no further
strongly connected subset of $N$ can be added while preserving nestedness. Maximal nested collections can be constructed iteratively using an \emph{activation sequence} as we now review.

\begin{definition} 
An \emph{activation sequence} is an ordered list $(i_1,\ldots,i_r)$ of distinct vertices, with associated nested collections $(\N_0,\ldots,\N_r)$ defined recursively as follows. Set $\N_0=\varnothing$. Subsequently, obtain $\N_a$ from $\N_{a-1}$ by adjoining the strongly connected component of $\Gamma[\{i_1,\ldots,i_a\}]$ containing $i_a$ as a new member.
\end{definition}
The support of $\N_a$ is $N_a=\{i_1,\ldots,i_a\}$ for $1\leq a\leq r$, and we set $N_0=\varnothing$. We say that the activation sequence \emph{yields} $\N$ if $\N_r=\N$.

\begin{lemma}[{\cite[Lemma~3.4]{LP2}}]
For any activation sequence $(i_1,\ldots,i_r)$, each
associated collection $\N_a$ is a maximal nested collection. Every maximal nested collection can be
obtained from an activation sequence.
\end{lemma}

The following property of an activation sequence will be
used in the positivity proof.

\begin{lemma}\label{lem:activation}
Given an activation sequence $(i_1,\ldots,i_r)$, every
strongly connected component of $\Gamma[N_a]$ belongs
to $\N_r$ for $0\leq a\leq r$.
\end{lemma}

\begin{proof}
The inclusion-maximal members of $\N_a$ are pairwise
disjoint and have union $N_a$. By nestedness, they are
precisely the strongly connected components of $\Gamma[N_a]$.
By construction, $\N_a\subseteq\N_r$, and the result follows.
\end{proof}

%%%%%%%%%%%%%%%%%%%%%%%%%%%%%%%%%%%%%%%%%%%%%%%%%%%%%%%%
%%%%%%%%%%%%%%%%%%%%%%%%%%%%%%%%%%%%%%%%%%%%%%%%%%%%%%%%
\section{Graph LP algebras}\label{sec:graph-lp}
We now turn to our main object of study, the graph Laurent phenomenon algebra $\A_\Gamma$ associated with the directed graph $\Gamma$. Its coefficient ring is $R=\Z[A_1,\ldots,A_n]$, where the $A_i$ are algebraically independent. We take $\A_\Gamma$ to be the normalized LP algebra with initial cluster variables $X_1,\ldots,X_n$, algebraically independent over $\operatorname{Frac}(R)$, and initial exchange polynomials
\[
F_i=A_i+\sum_{i\to j}X_j \qquad (i\in[n]).
\]
We restrict our attention to the characterisation of its cluster variables and clusters, which
provide all the structure needed for the positivity proof.
For a more detailed discussion, see~\cite{LP,LP2}.

The cluster variables of $\mathcal{A}_\Gamma$ are of two types: $X_i$ for $i\in[n]$,
and $Y_I$ for nonempty strongly connected subsets $I$.
The latter are defined through functions $Y_S$ associated
with arbitrary subsets $S\subseteq[n]$.

An \emph{acyclic function on $S$} is a map $f:S\to[n]$
such that, for every $i\in S$, either $f(i)=i$ or
$i\to f(i)$ is an edge of $\Gamma$, and the only directed
cycles among the arrows $i\mapsto f(i)$ are loops. Set $Y_\varnothing=1$. For nonempty $S\subseteq[n]$, the $Y_S$ functions are defined as
\[
   Y_S=\sum_{\substack{f:S\to[n]\\ f\text{ acyclic}}}
       \prod_{i\in S}\frac{\widetilde X_i(f)}{X_i},
   \qquad \qquad  \widetilde X_i(f)=
  \begin{cases}
    A_i, & f(i)=i,\\
    X_{f(i)}, & f(i)\ne i.
  \end{cases}
\]
The function $f(i)=i$ contributes
$\prod_{i\in S}A_i/X_i$, so $Y_S\ne0$ for every $S$.
\begin{lemma}[{\cite[Lemma~4.2]{LP2}}]
\label{lem:factorization}
Let $S\subseteq[n]$. Then $Y_S=\prod_{I\in\operatorname{SCC}(S)}Y_I.$
\end{lemma}
The characterisation of the clusters of $\mathcal{A}_\Gamma$ was one of the main results of~\cite{LP2}.
\begin{theorem}[{\cite[Theorem~1.1]{LP2}}]
The clusters of $\A_\Gamma$ are precisely $\x_{\N}=\{Y_I:I\in\N\}\cup\{X_i:i\notin N\}$,
where $\N$ ranges over maximal nested collections.
\end{theorem}

Combining Lemma~\ref{lem:activation} with
Lemma~\ref{lem:factorization} gives the following property
used in the positivity proof.

\begin{lemma}\label{cor:prefix-monomial}
Let $(i_1,\ldots,i_r)$ be an activation sequence yielding
$\N$. Then $Y_{N_a}$ is a monomial in $\x_{\N}$ for every
$0\leq a\leq r$.
\end{lemma}

\begin{proof}
By Lemma~\ref{lem:factorization} we have $Y_{N_a}=\prod_{I\in\operatorname{SCC}(N_a)}Y_I$. Lemma~\ref{lem:activation} gives
$\operatorname{SCC}(N_a)\subseteq\N$, so every factor
is a variable of $\x_{\N}$.
\end{proof}
\begin{remark}
The $Y_{N_a}$ will appear in denominators in the positivity proof. Since each $Y_{N_a}$ is a monomial in the variables of the chosen cluster, division by $Y_{N_a}$ preserves Laurent positivity.
\end{remark}
%%%%%%%%%%%%%%%%%%%%%%%%%%%%%%%%%%%%%%%%%%%%%%%%%%%%%%%%

\section{Two positive identities}
\label{sec:id}
We now introduce the two `$XY$' and `$YY$' identities which underlie the positivity proof. The $XY$-identity appeared already in~\cite{LP2}. The $YY$-identity proved below provides the crucial new ingredient for a straightforward inductive proof of positivity.

First we introduce some terminology. For vertices $i,j$ and a subset $S\subseteq[n]$, write
$\gamma:i\rightsquigarrow_S j$ for a simple directed path from $i$
to $j$ whose intermediate vertices lie in $S$. The endpoints need
not belong to $S$. The \emph{path polynomial} $P_S^{i,j}$ is defined as
\begin{equation}\label{eq:P-definition}
  P_S^{i,j}
  =\sum_{\gamma:i\rightsquigarrow_S j}Y_{S\setminus\gamma},
\end{equation}
where $S\setminus\gamma$ denotes the removal of all vertices of
$\gamma$ from $S$. If $i=j$, only the trivial path is allowed, so
$P_S^{i,i}=Y_{S\setminus\{i\}}$; in particular,
$P_S^{i,i}=Y_S$ when $i\notin S$. 
\begin{remark}
Each $P_S^{i,j}$ is a finite sum of $Y_B$ with $B\subseteq S$. This will be essential for the inductive positivity proof.
\end{remark}

We first recall two identities from~
\cite[Lemmas~4.7 and~4.13]{LP2}.
We refer to the first as the $XY$-identity.

\begin{lemma}[$XY$-identity]\label{lem:exit}
Let $S\subseteq[n]$ and $i\notin S$. Then
\begin{equation}\label{eq:exit}
  X_iY_{Si}
  =\sum_{j\notin Si}P_S^{i,j}X_j
   +\sum_{j\in Si}P_S^{i,j}A_j.
\end{equation}
\end{lemma}

\begin{lemma}\label{lem:local-path}
Let $S\subseteq[n]$, and let $i,j,k\in[n]$ be pairwise distinct
with $i,k\notin S$. Then
\begin{equation}\label{eq:local-path}
  Y_SP_{Sk}^{i,j}
  =Y_{Sk}P_S^{i,j}+P_S^{i,k}P_S^{k,j}.
\end{equation}
\end{lemma}

We will also use the following variant of
Lemma~\ref{lem:local-path}.

\begin{lemma}\label{lem:auxiliary-path}
Let $S\subseteq[n]$ and let $i,j,k\notin S$ be pairwise distinct.
Then
\begin{equation}\label{eq:auxiliary-path}
  Y_{Sj}P_{Sk}^{i,j}
  =Y_{Sjk}P_S^{i,j}
   +P_{Sj}^{i,k}P_S^{k,j}.
\end{equation}
\end{lemma}

\begin{proof}
Apply Lemma~\ref{lem:local-path} with base set $Sj$,
endpoints $i,j$, and added vertex $k$. The result follows
from $P_{Tj}^{i,j}=P_T^{i,j}$ for $i\ne j$, since a simple path
ending at $j$ cannot contain $j$ as an intermediate vertex,
and $j$ is removed from the set indexing its weight.
\end{proof}

The goal for the remainder of this section is to generalise~\cite[Lemma~4.11]{LP2}, which we now recall.
\begin{lemma}\label{lem:local-YY}
Let $S\subseteq[n]$ and let $i,j\notin S$ be distinct. Then
\begin{equation}\label{eq:local-YY}
  Y_{Si}Y_{Sj}
  =Y_{Sij}Y_S+P_S^{i,j}P_S^{j,i}.
\end{equation}
\end{lemma}

The generalisation of the above $YY$-identity requires the following intermediate result.

\begin{lemma}\label{lem:iterated-path}
Let $S\subseteq T\subseteq[n]$ and let $i,j\notin T$ be distinct.
Then
\begin{equation}\label{eq:iterated-path}
  Y_SP_T^{i,j}
  =Y_TP_S^{i,j}
   +\sum_{k\in T\setminus S}
      P_{T\setminus\{k\}}^{i,k}P_S^{k,j}.
\end{equation}
\end{lemma}

\begin{proof}
The proof follows by induction on $|T\setminus S|$. For $T=S$, the identity is immediate. Suppose $T\ne S$. Choose $k\in T\setminus S$ and set $T'=T\setminus\{k\}$. Let $R$ denote the right-hand side of~\eqref{eq:iterated-path} and consider the product $Y_{T'} R$. Separating the term indexed by $k$, we have
\[
  Y_{T'}Y_TP_S^{i,j}+Y_{T'}P_{T'}^{i,k}P_S^{k,j}
    +\sum_{\ell\in T'\setminus S}
       \left(Y_{T'}P_{T\setminus\{\ell\}}^{i,\ell} \right)P_S^{\ell,j}=Y_{T'}R.
\]
Applying Lemma~\ref{lem:auxiliary-path} to the terms in parentheses, with base set $T'\setminus\{\ell\}$, endpoints $i,\ell$, and added vertex $k$, gives
\[
\begin{aligned}
  Y_TY_{T'}P_S^{i,j}
       +Y_{T'}P_{T'}^{i,k}P_S^{k,j}+\sum_{\ell\in T'\setminus S}
       \Bigl(Y_TP_{T'\setminus\{\ell\}}^{i,\ell}
          +P_{T'}^{i,k}P_{T'\setminus\{\ell\}}^{k,\ell}\Bigr)
          P_S^{\ell,j}=Y_{T'}R.
\end{aligned}
\]
Grouping terms according to their common factors of $Y_T$ and $P_{T'}^{i,k}$ gives
\[
\begin{aligned}
  Y_T\left(
       Y_{T'}P_S^{i,j}
       +\sum_{\ell\in T'\setminus S}
          P_{T'\setminus\{\ell\}}^{i,\ell}P_S^{\ell,j}
       \right)+P_{T'}^{i,k}\left(
       Y_{T'}P_S^{k,j}
       +\sum_{\ell\in T'\setminus S}
          P_{T'\setminus\{\ell\}}^{k,\ell}P_S^{\ell,j}
       \right)=Y_{T'}R.
\end{aligned}
\]
Since $S\subseteq T'$ and
$|T'\setminus S|=|T\setminus S|-1$, the induction hypothesis with endpoints $(i,j)$ identifies
the first parenthesis as $Y_SP_{T'}^{i,j}$.  It also applies with endpoints $(k,j)$, because $k,j\notin T'$ and $k\ne j$. Hence the second parenthesis is $Y_SP_{T'}^{k,j}$. This gives
\[
  Y_S\bigl(Y_TP_{T'}^{i,j}+P_{T'}^{i,k}P_{T'}^{k,j}\bigr)=Y_{T'}R.
\]
Applying Lemma~\ref{lem:local-path} with base set $T'$
and added vertex $k$ identifies the expression in parentheses as $Y_{T'}P_T^{i,j}$. It follows that $Y_{T'}Y_SP_T^{i,j}=Y_{T'}R$. Cancelling
$Y_{T'}\ne0$ proves~\eqref{eq:iterated-path} and completes
the induction.
\end{proof}

We now prove the new $YY$-identity, which will provide the mechanism for the inductive proof of positivity in the next section.

\begin{lemma}[$YY$-identity]\label{prop:YY}
Let $S\subseteq T\subseteq[n]$ and $k\notin T$. Then
\begin{equation}\label{eq:YY}
  Y_TY_{Sk}
  =Y_{Tk}Y_S
   +\sum_{i\in T\setminus S}
       P_{T\setminus\{i\}}^{k,i}P_S^{i,k}.
\end{equation}
\end{lemma}

\begin{proof}
The proof follows by induction on $|T\setminus S|$. For $T=S$, the identity is immediate. Suppose $T \neq S$. Choose
$i\in T\setminus S$ and set $T'=T\setminus\{i\}$. Let $R$ denote the right-hand side of~\eqref{eq:YY} and consider the product $Y_{T'}R$.
Separating the term indexed by $i$, we have
\[
  Y_{T'}Y_{Tk}Y_S+Y_{T'}P_{T'}^{k,i}P_S^{i,k}
    +\sum_{\ell\in T'\setminus S} \left( Y_{T'}
       P_{T\setminus\{\ell\}}^{k,\ell} \right)P_S^{\ell,k}=Y_{T'}R.
\]
Applying Lemma~\ref{lem:auxiliary-path} to the term in parentheses, with base set $T'\setminus\{\ell\}$, endpoints $k,\ell$, and added vertex $i$, gives
\[
\begin{aligned}
  Y_{T'}Y_{Tk}Y_S
       +Y_{T'}P_{T'}^{k,i}P_S^{i,k}+\sum_{\ell\in T'\setminus S}
       \Bigl(Y_TP_{T'\setminus\{\ell\}}^{k,\ell}
          +P_{T'}^{k,i}P_{T'\setminus\{\ell\}}^{i,\ell}\Bigr)
          P_S^{\ell,k}=Y_{T'}R.
\end{aligned}
\]
Grouping terms with the common factor $P_{T'}^{k,i}$, we have
\[
\begin{aligned}
  Y_{T'}Y_{Tk}Y_S
       +Y_T\sum_{\ell\in T'\setminus S}
          P_{T'\setminus\{\ell\}}^{k,\ell}P_S^{\ell,k}+P_{T'}^{k,i}\left(
       Y_{T'}P_S^{i,k}
       +\sum_{\ell\in T'\setminus S}
          P_{T'\setminus\{\ell\}}^{i,\ell}P_S^{\ell,k}
       \right)=Y_{T'}R.
\end{aligned}
\]
Applying Lemma~\ref{lem:iterated-path} with sets $S\subseteq T'$ and endpoints $i,k$ identifies the expression in parentheses as $Y_SP_{T'}^{i,k}$. This gives
\[
\begin{aligned}
  Y_S\bigl(
       Y_{T'}Y_{Tk}+P_{T'}^{k,i}P_{T'}^{i,k}
       \bigr) +Y_T\sum_{\ell\in T'\setminus S}
       P_{T'\setminus\{\ell\}}^{k,\ell}P_S^{\ell,k}=Y_{T'}R.
\end{aligned}
\]
Using~\eqref{eq:local-YY}, with base set $T'$
and added vertices $i,k$, the term in parentheses is recognised as $Y_TY_{T'k}$. Therefore, we have
\[
  Y_T\left(
       Y_{T'k}Y_S
       +\sum_{\ell\in T'\setminus S}
          P_{T'\setminus\{\ell\}}^{k,\ell}P_S^{\ell,k}
       \right)=Y_{T'}R.
\]
Since $S \subseteq T'$ and $|T'\setminus S|=|T \setminus S|-1$, the expression in parentheses is simply $Y_{T'}Y_{Sk}$ by induction. It follows that $Y_TY_{T'}Y_{Sk}=Y_{T'}R$. Cancelling
$Y_{T'}\ne0$ proves~\eqref{eq:YY} and completes
the induction.
\end{proof}
%%%%%%%%%%%%%%%%%%%%%%%%%%%%%%%%%%%%%%%%%%%%%%%%%%%%%%%%%%%%
%%%%%%%%%%%%%%%%%%%%%%%%%%%%%%%%%%%%%%%%%%%%%%%%%%%%%%%%%%%%
%%%%%%%%%%%%%%%%%%%%%%%%%%%%%%%%%%%%%%%%%%%%%%%%%%%%%%%%%%%%

\section{Laurent positivity}\label{sec:positivity}
We now prove that every $Y_S$ for $S\subseteq[n]$ and every $X_i$ for $i\in[n]$ has a Laurent expansion with nonnegative coefficients in the variables of any chosen cluster. This resolves~\cite[Conjecture~7.1(2)]{LP2}. We formulate the result precisely below.

Fix a maximal nested collection $\N$ with support $N$, and
let $\x_{\N}$ be its associated cluster. Define the semiring
of Laurent polynomials in $\x_{\N}$ with coefficients in
$\Z_{\geq0}[A_1,\ldots,A_n]$ by
\begin{equation}\label{eq:positive-semiring}
  \LP=\Z_{\geq0}[A_1,\ldots,A_n]
       \bigl[Y_I^{\pm1}:I\in\N,\;
             X_i^{\pm1}:i\notin N\bigr].
\end{equation}
This semiring is closed under addition, multiplication and
division by monomials in $\x_{\N}$. 

\begin{theorem}[Laurent positivity]\label{thm:all-functions}
For every $S\subseteq[n]$ and every $i\in[n]$, both $Y_S$
and $X_i$ belong to $\LP$.
\end{theorem}

\begin{proof}
Choose an activation sequence $(i_1,\ldots,i_r)$ yielding
$\N$. We have $N_r=N$. By Lemma~\ref{cor:prefix-monomial}, $Y_{N_a}$ is a monomial in $\x_{\N}$ for every
$0\leq a\leq r$.

\medskip
\noindent\emph{Step 1: $Y$-functions supported in $N$.}
We prove by induction on $a$ that
\begin{equation}\label{eq:prefix-positivity}
  Y_S\in\LP\qquad\text{for every }S\subseteq N_a.
\end{equation}
For $a=0$, the only subset is $\varnothing$, for which
$Y_\varnothing=1$ and the claim is immediate. Suppose the claim holds for $N_{a-1}$,
and take $S\subseteq N_a$. If $i_a\notin S$, then
$S\subseteq N_{a-1}$ and $Y_S 
\in \LP$ by induction.
Otherwise, set $T=S\setminus\{i_a\}$, so that
$T\subseteq N_{a-1}$. Applying Lemma~\ref{prop:YY} to $T\subseteq N_{a-1}$ and the vertex $i_a$ gives
\begin{equation}\label{eq:positive-YY-recursion}
  Y_S=\frac{1}{Y_{N_{a-1}}}
       \left(
          Y_{N_a}Y_T
          +\sum_{j\in N_{a-1}\setminus T}
             P_{N_{a-1}\setminus\{j\}}^{i_a,j}
             P_T^{j,i_a}
       \right).
\end{equation}
Here $Y_T\in\LP$ by induction. As remarked below~\eqref{eq:P-definition},
every $Y_B$ occurring in
$P_{N_{a-1}\setminus\{j\}}^{i_a,j}$ has
$B\subseteq N_{a-1}\setminus\{j\}$, and every $Y_B$
occurring in $P_T^{j,i_a}$ has $B\subseteq T$.
Thus both path factors are sums of $Y$-functions supported
in $N_{a-1}$, and each belongs to $\LP$ by induction. Finally, $Y_{N_a}$ and $Y_{N_{a-1}}$ are monomials in
$\x_{\N}$. Therefore, the numerator of
\eqref{eq:positive-YY-recursion} belongs to $\LP$,
and division by $Y_{N_{a-1}}$ preserves this membership.
This proves~\eqref{eq:prefix-positivity}. Taking $a=r$,
we obtain
\begin{equation}\label{eq:support-positivity}
  Y_S\in\LP\qquad\text{for every }S\subseteq N. 
\end{equation}

\medskip
\noindent\emph{Step 2: the variables $X_i$.}
If $i\notin N$, then $X_i$ is a variable of the fixed
cluster. For $i\in N$, apply the $XY$-identity,
Lemma~\ref{lem:exit}, with $S=N\setminus\{i\}$. Since
$S\cup\{i\}=N$, this gives
\begin{equation}\label{eq:X-positivity-recursion}
  X_i=\frac{1}{Y_N}
       \left(
          \sum_{j\notin N}P_{N\setminus\{i\}}^{i,j}X_j
          +\sum_{j\in N}P_{N\setminus\{i\}}^{i,j}A_j
       \right).
\end{equation}
Every $Y_B$ occurring in these path polynomials is
supported in $N$, so each path polynomial belongs to
$\LP$ by~\eqref{eq:support-positivity}. Every $X_j$
in the first sum is a variable of $\x_{\N}$, and each
$A_j$ is an allowed coefficient. Also, $Y_N=Y_{N_r}$
is a monomial in $\x_{\N}$.
Hence~\eqref{eq:X-positivity-recursion} proves
\begin{equation}\label{eq:all-X-positive}
  X_i\in\LP\qquad\text{for every }i\in[n].
\end{equation}

\medskip
\noindent\emph{Step 3: all remaining $Y$-functions.}
We prove $Y_S\in\LP$ by induction on $|S\setminus N|$.
The base case, $S\subseteq N$,
is~\eqref{eq:support-positivity}. Otherwise, choose
$i\in S\setminus N$ and put $T=S\setminus\{i\}$.
Applying the $XY$-identity to $T$ and $i$ gives
\begin{equation}\label{eq:remaining-Y-recursion}
  Y_S=\frac{1}{X_i}
       \left(
          \sum_{j\notin S}P_T^{i,j}X_j
          +\sum_{j\in S}P_T^{i,j}A_j
       \right).
\end{equation}
For every $Y_B$ occurring in a path polynomial
$P_T^{i,j}$, we have $B\subseteq T$. Since $i\notin N$,
\[
  |B\setminus N|
  \leq |T\setminus N|
  =|S\setminus N|-1.
\]
The induction hypothesis therefore applies to every such
$Y_B$, and all the path polynomials
in~\eqref{eq:remaining-Y-recursion} belong to $\LP$.
Every $X_j$ belongs to $\LP$
by~\eqref{eq:all-X-positive}, so the numerator also
belongs to $\LP$. Since $i\notin N$, the denominator $X_i$ is a variable of the fixed cluster. Division by $X_i$ therefore
preserves membership in $\LP$. This completes the
induction and the proof.
\end{proof}

Every cluster variable is an $X_i$ or a $Y_I$ with $I$
strongly connected. Since $\N$ was arbitrary and every
cluster has the form $\x_{\N}$,
Theorem~\ref{thm:all-functions} establishes Laurent
positivity for $\A_\Gamma$.

\section*{Acknowledgements}
Funded by the European Union (ERC, UNIVERSE PLUS, 101118787). Views and opinions expressed are however those of the author(s) only and do not necessarily reflect those of the European Union or the European Research Council Executive Agency. Neither the European Union nor the granting authority can be held responsible for them. RG is supported by the Institute for Advanced Study and an anonymous donor. I would like to thank Thomas Lam and Pavlo Pylyavskyy.


\begin{thebibliography}{}

\bibitem[BCKZ]{BCKZ}
{\sc E.~Banaian, S.~Chepuri, E.~Kelley, and S.~W.~Zhang:}
Rooted clusters for graph LP algebras.
SIGMA Symmetry Integrability Geom. Methods Appl. 18 (2022),
Paper No.~089, 30~pp.

\bibitem[BCKZ2]{BCKZ2}
{\sc E.~Banaian, S.~Chepuri, E.~Kelley, and S.~W.~Zhang:}
Hydra graphs for graph LP algebras.
Ann. Comb. 30 (2026), 655--720.

\bibitem[DKW]{DKW}
{\sc G.~Z.~Dantas e Moura, R.~C.~Telekicherla Kandalam,
and D.~Woodruff:}
Cluster monomials in graph Laurent phenomenon algebras.
Algebraic Combinatorics 8 (2025), no.~4, 997--1019.

\bibitem[FZ]{CA1} {\sc S.~Fomin and A.~Zelevinsky:}
Cluster algebras. I. Foundations, {\sl J. Amer. Math. Soc.}, 15
(2002), no.~2, 497--529.

\bibitem[GM]{GM}
{\sc S.~S.~Gastineau and G.~Moreland:}
A binomial Laurent phenomenon algebra associated with the complete graph.
J. Algebraic Combin. 46 (2017), no.~3--4, 549--569.


\bibitem[LP1]{LP} {\sc T.~Lam and P.~Pylyavskyy:}
Laurent phenomenon algebras.
Cambridge J. Math. 4 (2016), no.~1, 121--162.

\bibitem[LP2]{LP2} {\sc T.~Lam and P.~Pylyavskyy:}
Linear Laurent phenomenon algebras.
Int. Math. Res. Not. IMRN 2016, no.~10, 3163--3203.


\end{thebibliography}
\end{document}